\documentclass[a4paper,11pt]{amsart}

\pdfoutput=1
\usepackage{graphicx}
\usepackage{wrapfig} 
\usepackage{amsmath,amsfonts,amssymb,amsthm}
\usepackage{amsopn}
\usepackage{mathrsfs}
\usepackage{verbatim}
\usepackage{url}
\usepackage{float}

\let\oldsqrt\sqrt
\def\sqrt{\mathpalette\DHLhksqrt}
\def\DHLhksqrt#1#2{\setbox0=\hbox{$#1\oldsqrt{#2\,}$}\dimen0=\ht0
  \advance\dimen0-0.2\ht0
  \setbox2=\hbox{\vrule height\ht0 depth -\dimen0}%
  {\box0\lower0.4pt\box2}}

\newlength{\LENGTH}
\def\sumprime{\mathop{\sum{\raise3pt\hbox{${}'$}}}}
\def\clap#1{\hbox to 0pt {\hss #1\hss}}

\newcommand{\im}{\mathrm{Im}}
\newcommand{\re}{\mathrm{Re}}
\newcommand{\mx}[1]{\begin{smallmatrix}#1\end{smallmatrix}}

\newcommand{\mphantom}[1]{\mathrel{\phantom{#1}}}

\newcommand{\Z}{\mathbb{Z}}
\newcommand{\Q}{\mathbb{Q}}
\newcommand{\R}{\mathbb{R}}

\newcommand{\F}{\mathcal{F}}

\newcommand{\UHP}{\mathbb{H}}

\def\SL{\operatorname{SL}}

\def\t{\tau}

\theoremstyle{plain}
\newtheorem{thm}{Theorem}[section] 
\newtheorem{prop}[thm]{Proposition}
\newtheorem{cor}[thm]{Corollary}
\newtheorem{lem}[thm]{Lemma}

\newtheorem*{thm*}{Theorem}
\newtheorem*{keylem*}{Key Lemma}

\theoremstyle{definition}

\theoremstyle{plain}

\begin{document}
\begin{abstract}

We improve the results from \cite{ElSe} about the Eisenstein series
$E_2(z)=1-24\sum_{n=1}^\infty \frac{nq^n}{1-q^n}$. In particular we show that there
exists exactly one (simple) zero in each Ford circle and give an approximation to its location.  \end{abstract}

\title{Estimates on the zeros of $E_2$}
\author{\"Ozlem Imamo\=glu, Jonas Jermann, {\'A}rp{\'a}d T{\'o}th}
\address{Department of Mathematics, ETH Zurich}
\email{ozlem@math.ethz.ch}
\thanks{ \"O. Imamoglu and J. Jermann are supported by SNF 200020-144285}
\address{Department of Mathematics, ETH Zurich}
\email{jjermann@math.ethz.ch}
\address{Eotvos Lorand University, Budapest}
\email{toth@cs.elte.hu}
\date{\today}
\maketitle
\section{Introduction}
Let ${\mathcal M}_k(\Gamma)$ be the space of holomorphic modular forms of weight $k$ for the full modular group $\Gamma=PSL(2,\Z)$. It is well known that ${\mathcal M}_k(\Gamma)$ has dimension $\frac{k}{12} +O(1)$ and a modular form $f\in {\mathcal M}_k$ has $\frac{k}{12} +O(1)$ inequivalent zeros in  a fundamental domain $\Gamma\backslash\UHP.$

 For the cuspidal Hecke eigenforms, it is   a consequence of the recent proof  of the holomorphic Quantum Unique Ergodicity(QUE)   by  Holowinsky and Soundararajan  \cite{HS}  that the zeros are uniformly distributed.
More precisely, 
 for    a sequence $\{f_k\}$ of cuspidal Hecke eigenforms of weight $k$ we have that 
  as $k\rightarrow\infty$ the zeros of $f_k$ become 
equidistributed with respect to the normalized hyperbolic measure $\frac{3}{\pi}\frac{ dx dy}{ y^2}.$

In contrast to this  in the  case   of  Eisenstein series, it was conjectured by R.A. Rankin in 1968 and proved by F.K.C. Rankin and Swinnerton-Dyer \cite{RS-D} that all the zeros, in the standard fundamental domain, of  the series
$$E_k(\t)=\frac{1}{2}\sum_{(c,d)=1} (c\t+d)^{-k}$$
lie on the geodesic  arc $\{z\in\UHP\ : \ |z|=1, 0\leq\re{z}\leq 1/2\}$
and as $k\rightarrow\infty$  they become uniformly distributed on this unit arc. A similar result for the cuspidal Poincare series was proved by R.A. Rankin \cite{Ra}. For generalizations of these results to other Fuchsian groups and  to weakly holomorphic modular functions   see   \cite{AKN}, \cite{DJ},  among many others. 
For some recent work on the zeros of holomorphic Hecke cusp forms that lie on the geodesic segments of the standard fundamental domain see \cite{GS}.

Next we turn our attention to  the location of the zeros of the non-modular   Eisenstein series of weight 2
  \begin{align}
E_2(\t)=1-24\sum_{n=1}^\infty\sigma(n)q^n=1-24\sum_{n=1}^\infty
\frac{nq^n}{1-q^n},\ q=e^{2\pi i\t}.
\end{align}

 As it is well known $E_2(\t)$ is not  modular but it  is   a quasimodular form and
 for $\gamma=\left(\mx{a&b\\c&d}\right)\in \SL_2(\Z)$ satisfies the  transformation property  
 \begin{align}
E_2(\gamma \t)(c\t+d)^{-2}=E_2(\t)+\frac{6}{\pi i}\frac{1}{(\t+\frac{d}{c})},  
\label{E2_trafo}
\end{align}

The zeros  of $E_2$ has already been investigated in \cite{ElSe}.
El Basraoui and Sebbar showed that there are infinitely many non-equivalent
zeros of $E_2$ and two zeros are equivalent if and only if one is a $\Z$-translate
of the other. Of particular interest is the unique zero on the imaginary axis $x=0$
and the unique zero on $x=\frac{1}{2}$. Those two zeros also occur in \cite{SeFa}. 
They were computed by H. Cohen:
\begin{align}
\t_0&=0.52352170001799926680053440480610976968 \dots i\label{E2_z0_zero}\\
\t_{\frac{1}{2}}&=\frac{1}{2}+0.13091903039676244690411482601971302060\dots i\label{E2_z12_zero}
\end{align}

If
\begin{align*}
\mathcal{F}&:=\left\lbrace z\in\UHP\ \middle|\ |z|>1,\ |\re(z)|<\frac{1}{2}\right\rbrace
\cup \left\lbrace z\in\UHP\ \middle|\ |z|\ge 1,\ -\frac{1}{2}\le \re(z)\le 0\right\rbrace
\end{align*}
is the standard (strict) fundamental domain, El Basraoui and Sebbar also show that there are
infinitely many $SL_2(\Z)$-translates of $ \mathcal{F}$ that contain a zero and infinitely many
that do not contain a zero.  

In this note we  improve these results.   More precisely 
  for $(a,c)=1$ recall that the associated Ford circle is the circle on the upper
half plane with center $\frac{a}{c}+\frac{1}{2c^2}i$ and radius
$\frac{1}{2c^2}$ .
Then we have

\begin{thm}\label{maint}
Inside each Ford circle there is a unique (simple) zero of $E_2$ and $E_2$
has no other zeros on the upper half plane.

Moreover in the Ford circle at the cusp
$\frac{a}{c}\in\Q$ the zero satisfies the following approximation:
\begin{align} 
0.000075\frac{1}{c^2}<\left|\t-\frac{a}{c}-\frac{\pi}{6c^2}i\right| < 0.0000777\frac{1}{c^2}.\label{bdH}
\end{align}
\end{thm}

Since $\pi/6$ is close to $1/2$, the zeros are almost at the center of each Ford circle and in fact when plotted on a not small enough scale  as in   the following picture,  the zeros seem to have remarkable uniformity on the upper half plane.

\begin{figure}[h!]
\centering
\includegraphics[width=0.9\textwidth]{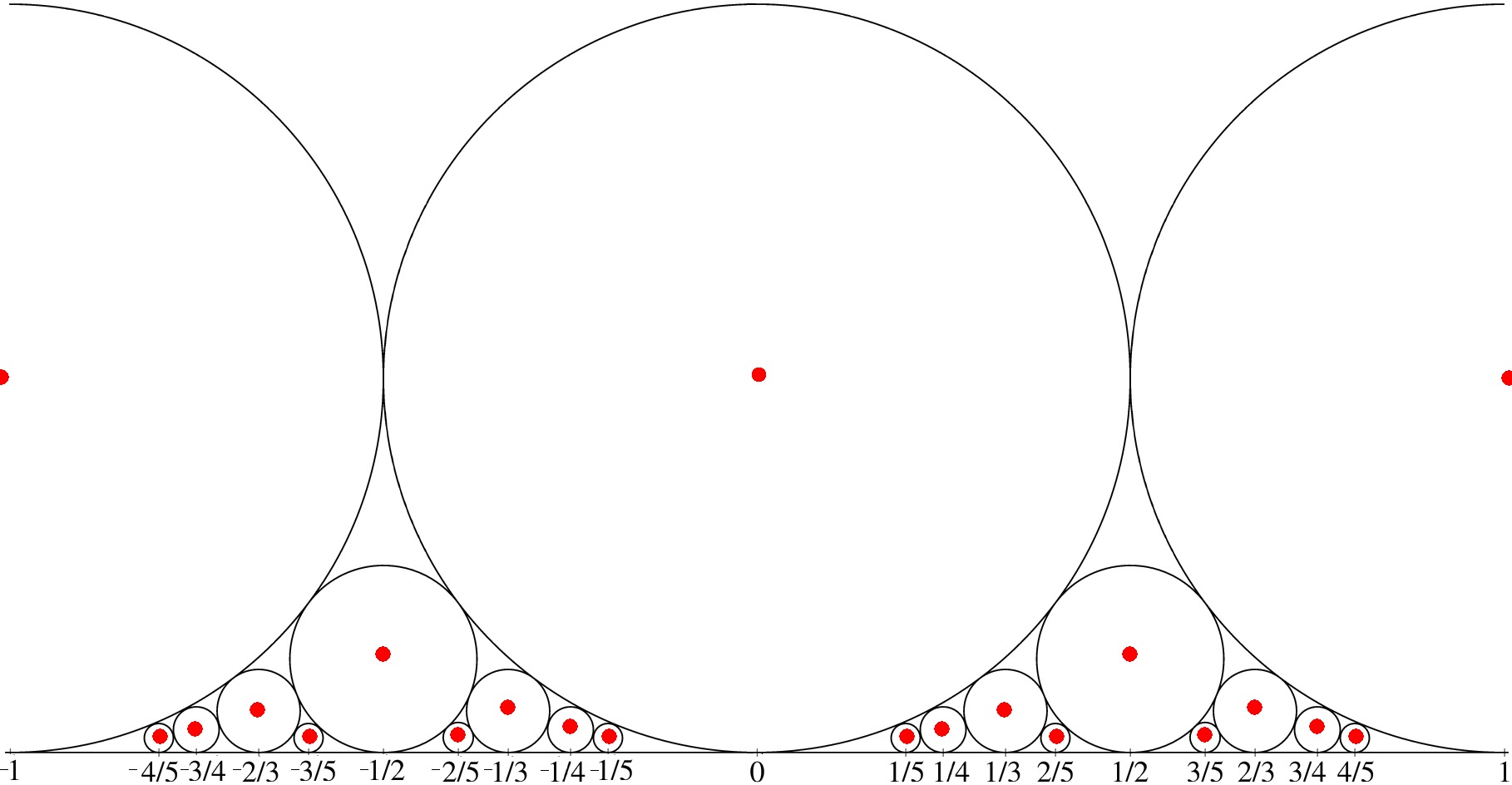}
\end{figure}
%
Theorem \ref{maint} restricts the possible location of the zeros to a very thin
annular region in each Ford circle. For  the images of these zeros  inside the standard fundamental domain $\F$  we have

\begin{thm}\label{main2}
Let $\t=\gamma z$ be the unique zero of $E_2$ around $\frac{a}{c}$  with
$\gamma=\left(\mx{a&b\\c&d}\right)$,  $c>0$   and
$z\in\mathcal{F}$ then we have
\begin{align}  
0.00027&<\left|z-\left(-\frac{d}{c}+\frac{6}{\pi}i\right)\right|<0.00029.\label{bdF}
\end{align}
 \end{thm}

We may rephrase Theorem \ref{maint} as saying that around each finite cusp there is a unique
zero of $E_2$. i.e. given $(a,c)=1$ with $c>0$, there is a unique matrix
$\gamma=\left(\mx{a&b\\c&d}\right)\in SL_2(\Z)$ such that
$\gamma\mathcal{F}$ contains a zero.

While  preparing this manuscript we learned  that 
  R. Wood   and  M. Young \cite{YW} proved a quantitatively equivalent result about the locations
of the zeros of $E_2$ simultaneously and independently of us. 

It is
possible to give further improvements on these approximations by
expanding either in Taylor or Fourier series  but we will not pursue this
in this paper at any depth.
   In the next section we give the proofs of Theorem \ref{maint} and Theorem \ref{main2} and in the final section we provide some numerical data on the Fourier and Taylor expansions.

\section{Proofs}
The proofs of Theorem \ref{maint} and Theorem \ref{main2} are  based on the mapping properties of the function  $f(z)$ defined by

\begin{equation}\label{f}
f(z):=z-\frac{6i}{\pi E_2(z)}.
\end{equation}
We will see that    $f$ is analytic in the fundamental domain, and even on $\{z: \im z> 0.53\}$.  Let $z$ be in the fundamental domain $\F$, and let $\gamma =\left( \begin{smallmatrix} a&b\\c&d\end{smallmatrix} \right)
 \in SL_2(\Z)$. Then based on the transformation formula of $E_2$, one has  the   fundamental relation  
\begin{equation}\label{fandE}
E_2(\gamma z)=0 \iff   f(z)=-d/c
\end{equation}

Therefore we are interested in the solutions of the equation
$$
f(z)=-d/c \quad\quad z \in \F.
$$

Here $d,c \in {\mathbf Z}$. Let's start with the solution of
$$
f(z) \in \R \quad\quad z \in \F.
$$ 
Let $f=u+iv$ and
\begin{equation}
A=\{ z\in \overline{\F}: v(z)=0\}.\label{def_A}
\end{equation}

As a simple application of the implicit function theorem it will be proved below that $A$ is  the graph of a function $\phi: [-1/2,1/2] \to (0,\infty)$. More precisely  if $x+iy \in \F$ and $v(x,y)=0$ then $y=\phi(x)$. 

In view of  the relation (\ref{fandE})
the possible values $c,d\in {\mathbf Z}$ that arise from a zero  of $E_2$ are those for which $-d/c\in f(A)$. Let $\t=\gamma z$ be a zero of $E_2$ with   $\gamma =\left( \begin{smallmatrix} a&b\\c&d\end{smallmatrix} \right)
 \in SL_2(\Z)$ and  $z \in \F$. Then we will show that $|d/c|\leq 1/2$ and conversely any such $-d/c\in [-1/2,1/2]$ arises from a zero of $E_2(z)$.

Namely we have 

\begin{thm} \label{f(A)}

$f(A)=[-1/2,1/2]$. \end{thm}

The first claim in Theorem \ref{maint} now follows immediately from 

\begin{cor}\label{cor1}
If $z \in \F$, $\gamma=\left( \begin{smallmatrix} a&b\\c&d\end{smallmatrix} \right) \in SL_2(\Z)$ and $E_2(\gamma z)=0$ then $|d/c|\leq 1/2$.
\end{cor}
In particular Corollary \ref{cor1} uniquely determines the
fundamental domain around $\frac{a}{c}$ which contains the zero of $E_2$.
%

The more precise location of zeros of $E_2$ given in   (\ref{bdH}) and  (\ref{bdF}) follow from bounds for $E_2$ and its derivative. 
%
%
%

To prove  Theorem \ref{f(A)}  
we will use the implicit function theorem. To ease the  notation let $f=u+iv$. We are interested in $v(x,y)=0$. We need the following estimates. 

\begin{lem} \label{lem-bdE}Let $z=x+iy\in \UHP$ and let $q=e^{2\pi iz}$.
We have the following estimates. 
\begin{align} 
|E_2(z)-1| &\le\frac{24|q|}{(1-|q|)^3},\label{bdE}\\
23.48|q| &< |E_2(z)-1| \quad \text{ for $y\ge\frac{\sqrt{3}}{2}$},\\
|E_2'(z)| &\leq 48\pi \frac{|q|(1+|q|)}{(1-|q|)^5} \label{bdEp}
\end{align}
\end{lem}

In particular for $y>0.53$ we have $0<|E_2(z)-1|<1$ and so $E_2(z)\neq 0$ in this region and hence   $f$ in (\ref{f}) is well defined. 


\begin{proof}
We have 
$$
|E_2(w)-1|\le 24\sum_{n=1}^\infty \frac{n|q|^n}{1-|q|^n}\le\frac{24}{1-|q|}\sum_{n=1}^\infty n|q|^n=\frac{24|q|}{(1-|q|)^3}.
$$

The claim for the lower bound for $E_2(z)$ and $E_2'(z)$ follows along the same lines.
\end{proof}

\begin{lem} \label{bdfp}
If $\im z\geq 1$ then 
\begin{equation} |f'(z)-1|<0.57
\end{equation}
\end{lem}

\begin{proof}
Since 
$$
|f'(z)-1|=\frac{6|E_2'(z)|}{\pi |E_2(z)|^2}
$$
the result follows from Lemma \ref{lem-bdE} and the bounds (valid for $\im z\geq 1$)
$$
|E_2(z)|\geq |E_2(i\im z)|\geq |E_2(i)| 
$$
and
$$
|E_2'(z)|\leq |E_2'(i\im z)|\leq |E_2'(i)|.
$$

\end{proof}

We will also need the following

\begin{lem}\label{lem-Imf}
For $x\in [-1/2,1/2]$ and $y \in [\sqrt{3}/2,1]$ we have
$$
\im f(x+iy)<0.
$$
\end{lem}

\begin{proof}
For any $x$ we have 
$$
\re E_2(x+iy) \geq 1-\frac{24e^{-2\pi y}}{(1-e^{-2\pi y})^3}
$$
and
$$
|E_2(x+iy)|\leq 1+\frac{24e^{-2\pi y}}{(1-e^{-2\pi y})^3}
$$

Therefore for $y \in [\sqrt{3}/2,1]$ we have
$$
\im f(x+iy) = y - \frac{6 \re E_2(x+iy)}{\pi |E_2(x+iy)|^2} < 1- \frac{6(1-t)}{\pi (1+t)^2}  
$$
where 
$$t=\frac{24e^{-2\pi \sqrt{3}/2}}{(1-e^{-2\pi \sqrt{3}/2})^3}.$$

The lemma follows from the numerical estimate $t<0.1054$.  
\end{proof}
Since $\partial_y v= \partial_x u = \re f'(z)$ we also have

\begin{lem}\label{prop-monotone}
If $y\geq 1$ then $\partial_y v(x,y)>0$.
\end{lem}

\begin{proof}
By Lemma \ref{bdfp},
$$
|f'(z)-1|<0.57.
$$
So $\partial_yv(z) = \re f'(z)>0.4$.
\end{proof}

From Lemma  \ref{prop-monotone} it follows that for each $x \in [-1/2,1/2]$ there can be only one $y$ such that $v(x+iy)=0$. By
Lemma \ref{lem-Imf}, it is also clear that such a $y$ exists since $\lim_{y\to \infty} v(x+iy)=\infty$. Therefore there is a function $\phi: [-1/2,1/2] \to (0,\infty)$ such that $x+iy \in \mathcal{F}$ and $v(x,y)=0$ implies $y=\phi(x)$.

Moreover this $\phi$ is differentiable by the implicit function theorem, and we have $\phi'(x)=-\partial_x v/\partial_y v$. 

Let now the arc $A$ be as in (\ref{def_A}). We know from El Basraoui and
Sebbar that there are points $z_{1/2}$ and $z_{-1/2}$ on the both of the vertical boundaries
of $\mathcal {F}$, where $f$ takes the values $-1/2$ and $1/2$. They correspond to the zeros of $E_2$ with real parts $1/2$ and $-1/2$. 

By the intermediate value theorem it follows that $f$ restricted to the arc $A$, (which is the same as $u$ since $v$ vanishes on $A$) takes all values between $-1/2$ and $1/2$. This part of the claim shows that there is a zero in each Ford circle. To prove that there is only one zero in each Ford circle, we need to show that $f(A)=u(A)=[-1/2,1/2]$. This we prove by showing that $u$ decreases along $A$ as $x$ increases from $-1/2$ to $1/2$. 

\begin{prop}
Let $f=u+iv$ and $\phi:[-1/2,1/2]\to (1,\infty)$ the function such that 
$$
\{z \in \F: v(z)=0\} = \{ (x,\phi(x)) : x \in [-1/2,1/2]\}
$$
For $|x|<1/2$ we have $$\frac{d}{dx} u(x,\phi(x))>0.$$
\end{prop}

\begin{proof}
We need to show that 
$$
\frac{d}{dx} u(x,\phi(x)) = \partial_x u(x,\phi(x)+\partial _y u(x,\phi(x))\phi'(x)>0.
$$ 

By implicit differentiation of $v(x,\phi(x))=0$ we have $\phi'(x)=-\partial_x v/\partial_yv$. Therefore 

$$
 \frac{d}{dx} u(x,\phi(x)) = \frac{\partial_x u \partial_y v-\partial_y u\partial_xv}{\partial_yv}  = \frac{(\partial_xu)^2+(\partial_y u)^2}{\partial_yv} = \frac{|f'|^2}{\partial_yv} > 0.
$$

\end{proof}
This also finishes the proof of Theorem \ref{f(A)} and hence the first claim in Theorem  \ref{maint}.

The next Proposition on the other hand proves the estimates (\ref{bdH}) and (\ref{bdF}) about the location of zeros in Theorem \ref{maint} and Theorem \ref{main2}. 
%
%
%
\begin{prop}\label{better_est}
If $E_2(\gamma z)=0$, $z=x+iy\in {\mathcal{F}}$, 
$\gamma=\left(\mx{a&b\\c&d}\right)\in SL_2(\Z)$ with $c>0$  and $\t=\gamma z$ then 
$1.909<y<1.91$ and
\begin{align}
0.000144&<|E_2(z)-1|<0.000149,\label{e2_est2}\\
0.00027&<\left|z-\left(-\frac{d}{c}+\frac{6}{\pi}i\right)\right|<0.00029,\\
0.0000750\frac{1}{c^2}&<\left|\t-\left(\frac{a}{c}+\frac{\pi}{6c^2}i\right)\right|< 0.0000777\frac{1}{c^2}.
\end{align}
\end{prop}

\begin{proof}
 We first note that since $y\ge\frac{\sqrt{3}}{2}$ on   $\overline{\mathcal{F}}$, using (\ref{bdE}) we have 
\begin{align}
|E_2(z)-1|&\le\frac{24}{\left(1-e^{-\pi\sqrt{3}}\right)^3}|q|< 24.32|q| < 0.106.
\end{align}
 And similarly 
 \begin{align}
23.48|q|< |E_2(z)-1|\label{eq_low_w_est}
\end{align}

  We next recall that $E_2(\gamma z)=0$ gives    $f(z)=-d/c$ which in return implies $z-\frac{6i}{\pi E_2(z)}=\frac{-d}{c}.$ Hence

\begin{align}
\left|z-\left(-\frac{d}{c}+\frac{6}{\pi}i\right)\right|&=\frac{6}{\pi}\frac{|E_2(z)-1|}{ |E_2(z)|}
\le \frac{6}{\pi}\frac{24|q|}{(1-|q|)^3\left(1-\frac{24|q|}{(1-|q|)^3}\right)^2}\\
&<\frac{144}{\pi(1-27|q|)^2}|q|\label{eq1}\\
&<\frac{6}{\pi}\frac{0.106}{(1-0.106)^2}<0.26
\end{align}
This in turn means that $y>\frac{6}{\pi}-0.26>1.64$ which gives a better
estimate for (\ref{eq1}). Repeating this procedure twice already  gives the
nice estimates:
\begin{align}
y&\ge 1.909,\quad |q|\le 6.18 \times10^{-6},\\
|E_2(z)-1| &< 24\frac{1}{(1-|q|)^3}|q| < 0.000149,\\
\left|z-\left(-\frac{d}{c}+\frac{6}{\pi}i\right)\right|&<0.00029.
\end{align}
The estimates also show that $y<\frac{6}{\pi}+0.00029<1.91$ resp. (using
(\ref{eq_low_w_est})) $|E_2(z)-1|>0.000144$, so
\begin{align}
0.00027< \left|z-(-\frac{d}{c}+\frac{6}{\pi}i)\right|.
\end{align}
In particular we also have (\ref{e2_est2}).
On the other hand  using $\t=\gamma^{-1}z$  and $E_2(\gamma z)=0$ we have 
\begin{align}
 E(z)=E(\gamma^{-1}\t)=  \frac{6}{\pi i}(-c)(-c\t+a) 
\end{align}
 and 
 $$\t=\frac{a}{c}+\frac{\pi}{6c^2}i+\frac{\pi}{6c^2}i(E_2(z)-1)$$
So
\begin{align}
\left|\t-\frac{a}{c}-\frac{\pi}{6c^2}i\right|&< \frac{\pi}{6c^2}\left|E_2(z)-1\right| < 0.0000777\frac{1}{c^2},
\end{align}
where we used (\ref{e2_est2}). Similarly we get:
\begin{align}
0.0000750\frac{1}{c^2}<&\frac{\pi}{6c^2}\left|E_2(z)-1\right|<\left|\t-\frac{a}{c}-\frac{\pi}{6c^2}i\right|.
\end{align}
\end{proof}

\section{Numerical calculations}
In this section we provide   additional numerical data on the zeros of $E_2(z)$.

 Let $F(x)=f^{-1}(x)$ denote the inverse of $f(z)$ restricted to $\R$. (Note that $F$
defines an (injective) holomorphic map in a neighbourhood of $\R$.)


Recall the $\SL_2(\Z)$-translates of
the zeros of $E_2$ inside the strict fundamental domain $\mathcal{F}$
are exactly given by $F(x)$ for $x\in\Q\cap\left(-\frac{1}{2},\frac{1}{2}\right]$.

\subsection{Fourier expansion}
\ \\
Let $H(x):=F(x)-x$. Since $F(x+1)=F(x)+1$, $H(x)$ is $1$-periodic, so it
possesses a Fourier expansion:
\begin{align}
H(x)=\sum_{n\in\Z}c_n\exp(2\pi i nx).
\end{align}

To get a formal expression for it we use the standard change of coordinates $q=e^{2\pi i \tau}$. Then $f(\tau)=\tilde{f}(q)$, $H(\t)=f^{-1}(\t)-\t= \frac{1}{2\pi i}(\tilde{H}\circ q)(\t)$ where
 
\begin{equation}
\tilde{H}(q)=\left(\log(\tilde{f}^{-1}(\log(q)))-\log(q)\right).
\end{equation}

Note that although $\tilde{f}$ has a term involving $\log q$ this cancels out and in fact we have that 
\begin{multline}
2\pi iH\left(\t-\frac{6i}{\pi}\right)=\tilde{H}\left(q\exp({12})\right)=
\\
\log\left(\frac{\left(q\exp\left(\frac{12}{E_2(q)}-12\right)\right)^{-1}}{q\exp(12)}\right)=\\
-12+\log\left(\frac{\left(q\exp\left(\frac{12}{E_2(q)}-12\right)\right)^{-1}}{q}\right)\in\Q[[q]].
\end{multline} 

The existence of the last formal power series and the rationality of its coefficients follow from the facts that $\frac{12}{E_2(q)}-12 \in \Q[[q]]$ and is divisible by $q$ in this ring.

The coefficients of $\tilde{H}(q\exp(12))$  can be computed explicitely and we have
\begin{align*}
\tilde{H}\left(qe^{12}\right)=-12 - 288q + 75168q^2 - 29321856q^3 + 13541649696q^4 + O(q^5)
\end{align*}
(Note that higher index coefficients no longer all lie in $\Z$.) 

From this the constants $c_n$ can be expressed, for example $c_0=\frac{6i}{\pi}$ and
$c_1=144\frac{e^{-12}}{\pi}i$, and $c_n=r_n \frac{e^{-12n}}{\pi} i$ for some rational number $r_n$. 

Calculations of the Fourier coefficients by numerical integration show perfect agreement with these formal computations:

The numerical results for $c_n$, $n\ge 0$ were calculated using sage:
\begin{align*}
c_0&=\mphantom{-}1.90985931710274402922660516047017234441351574888547738497i\\
c_1&=\mphantom{-}0.00028162994902227980400370919939063856289594529890275357i\\
c_2&=-0.00000045163288929282012635455207577614911204985274433204i\\
c_3&=\mphantom{-}0.00000000108245596925811696405920054080771200657423178116i\\
c_4&=-0.00000000000307154282808538137128721799597366123291772057i\\
c_5&=\mphantom{-}0.00000000000000957094344711630129941209014246199967789040i\\
c_6&=-0.00000000000000003165503372449709626701121359204401518804i\\
c_7&=\mphantom{-}0.00000000000000000010911333723210259127123321555755721374i\\
c_8&=-0.00000000000000000000038769575689989675972304397185016487i\\
c_9&=\mphantom{-}0.00000000000000000000000140991650336337176376718140570072i\\
c_{10}&=-0.00000000000000000000000000522238601930043508518075691844i
\end{align*}

Below is a picture of $H(x)$ for $x\in\Q\cap\left[-\frac{1}{2},\frac{1}{2}\right]$,
$x=\frac{d}{100}$ with $(d,100)=1$ (recall that this is the error term in the approximation $F(x)\approx x$).
\begin{figure}[h!]
\centering
\includegraphics[width=0.65\textwidth]{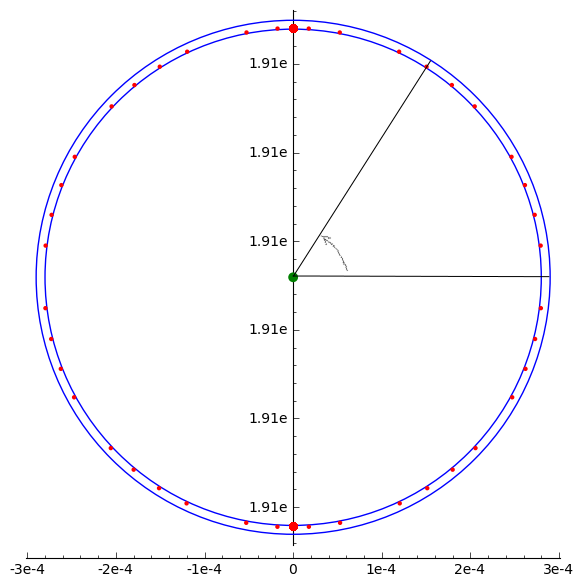}
\end{figure}

The closed curve $H(\R)=H\left(\left[-\frac{1}{2},\frac{1}{2}\right]\right)$
is already extremely well approximated by the first two terms $c_0+c_1\exp(2\pi ix)$.
Indeed estimate (\ref{bdF}) for $x=\frac{d}{c}$ (in fact for any $x\in\R$) and
$H\left(\frac{d}{c}\right)=F\left(\frac{d}{c}\right)-\frac{d}{c}$ shows:
\begin{align}
0.00027 < \left|H(x)-\frac{6i}{\pi}\right| < 0.00029.
\end{align}
These circular bounds are also drawn in the picture.

A good bound on the remaining Fourier coefficients $c_n$, $n\ge 2$ would
give an estimate on the \textit{angle} of $H(x)-c_0$.
Namely that the angle is \textit{close} to 
$
2\pi\left(x+\frac{1}{4}\right).
$ 
Note in particular that $H(0)\in i\R$ is located on the top of $H(\R)$, so
the angle of $H(0)-c_0$ is exactly $2\pi\frac{1}{4}$.
Similarly $H\left(\frac{-1}{2}\right)\in i\R$ is located on the bottom of
$H(\R)$ and the angle of $H\left(\frac{-1}{2}\right)-c_0$ is exactly
$2\pi\frac{-1}{4}$. That's both under the assumption that $c_0\in i\R$.

\subsection{Taylor expansion}
\ \\
We can also examine the Taylor expansion $F(x)=\sum_{n=0}^\infty b_n x^n$ around $x=0$.

By Lagrange inversion theorem the Taylor coefficients can be calculated in terms of the Taylor coefficients 
of $f(z)$ at $ z_0$, where $z_0$ is the
$SL_2(\Z)$-translate in the struct fundamental domain of the zero of $E_2$
on the imaginary axis.

Note that $f(z)$ is a rational function in $E_2(z)$ and $z$. It's derivatives can be calculated as rational functions in $E_2(z)$, $E_4(z)$, $E_6(z)$ and $z$. Therefore the Taylor expansion at $z_0$  can be expressed as
a rational function in $X:=E_4(z_0)$, $Y:=E_6(z_0)$, $Z:=E_2(z_0)$
and $z_0$. By the transformation property (\ref{E2_trafo}) we have $z_0=-\frac{12}{2\pi i Z}$, so in fact we get a rational function in $X$, $Y$ and $Z$.

For the first coefficient we have
$b_0=\left(\frac{6i}{\pi}\right)\frac{1}{Z}$.
For the remaining coefficients ($n>0$)  calculations seem to indicate that
$b_n$ is of the following shape:
\begin{align}
b_n = (-1)^n\left(\frac{6i}{\pi}\right)^{1-n}X^{-2n+1}a_n(X,Y,Z),
\end{align}
where $a_n(X,Y,Z)\in\Q[X,Y,Z]$ is a (homogenous) polynomial of degree
$10n-6$,
where $\deg(X):=4$, $\deg(Y):=6$, $\deg(Z):=2$.
The degree of the (homogenous) rational function $b_n(X,Y,Z)$ is $2(n-1)$.

Here is a list of the first few  polynomials $a_n(X,Y,Z)$ for $n>0$:
\begin{align*}
a_1 &= Z^2\\
a_2 &= XZ^5 + X^2Z^3 - 2YZ^4\\
a_3 &= X^2Z^8 - 2X^3Z^6 - 4XYZ^7 + X^4Z^4 - 4X^2YZ^5 + 8Y^2Z^6\\
a_4 &= X^3Z^{11} - 5X^4Z^9 - 6X^2YZ^{10} - 9X^5Z^7 + 20X^3YZ^8\\
    &  \quad + 20XY^2Z^9 + X^6Z^5 - 6X^4YZ^6 + 24X^2Y^2Z^7 - 40Y^3Z^8
\\
a_5 &= X^4Z^{14} - 8X^5Z^{12} - 8X^3YZ^{13} + \frac{66}{5}X^6Z^{10}\\
    &  \quad + 56X^4YZ^{11} + 36X^2Y^2Z^{12} - 20X^7Z^8 + 104X^5YZ^9\\
    &  \quad - \frac{836}{5}X^3Y^2Z^{10} - 112XY^3Z^{11} + X^8Z^6 - 8X^6YZ^7\\
    &  \quad + 48X^4Y^2Z^8 - 160X^2Y^3Z^9 + 224Y^4Z^{10}
\end{align*}


The numerical evaluations of those polynomials ($Z=E_2(z_0)$, etc.) give for Taylor coefficients   $b_n$  :
\begin{align*}
b_0 &= \mphantom{-}1.9101404964982709820376545357984830913777487030i\\
b_1 &= -0.9982361219015924374815710878280361648431190825\\
b_2 &= \mphantom{-}0.0055236842011260453610739166397990326586651337i\\
b_3 &= -0.01149489150274208316313259093815041563703067326\\
b_4 &= \mphantom{-}0.0178252611095229253133162329291348589135823077i\\
b_5 &= \mphantom{-}0.0218243134639575211728774441381952550676521991\\
b_6 &= \mphantom{-}0.0216634844629385759461618124642355350591382590i\\
b_7 &= -0.0173461622715009362175946164162223267332088060\\
b_8 &= \mphantom{-}0.0104172812309514952250501361120571266678673695i\\
b_9 &= \mphantom{-}0.0029825116383882005761911965146832408919333302
\end{align*}

For the zeros of $E_2$ near $1/c$ found in \cite{ElSe} by El Basraoui and Sebbar, the Taylor polynomials give a fast converging asymptotic series in $c$. In general the Taylor polynomial is inferior to the Fourier approximation and starts to give a reasonably good fit only after $n=8$.

Below is a picture of the Taylor approximation for $n=8$ together with plots
of $F(x)$ for $x\in\Q\cap\left[-\frac{1}{2},\frac{1}{2}\right]$,
$x=\frac{d}{100}$, $(d,100)=1$.

\begin{figure}[h!]
\centering
\includegraphics[width=0.65\textwidth]{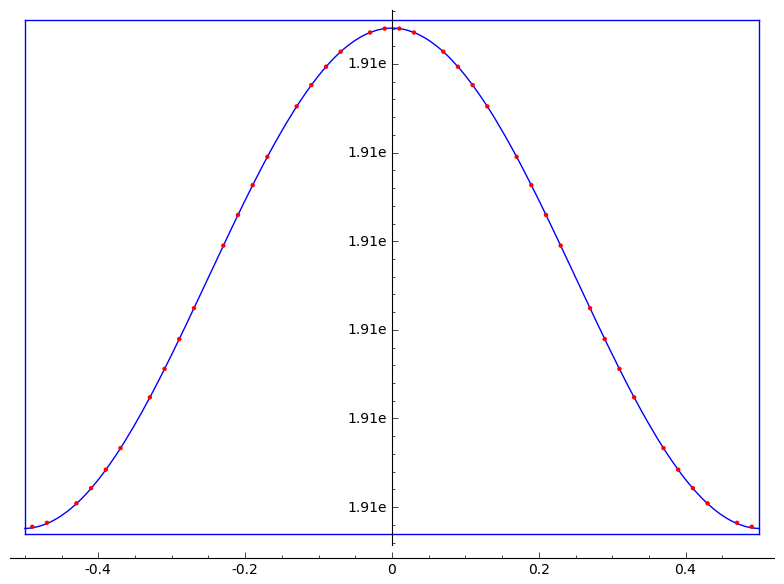}
\end{figure}
The plot also shows the   bounds from (\ref{bdF}).
\newpage

\bibliographystyle{alpha}
\bibliography{literature}
\end{document}